\documentclass[12pt]{article}
\usepackage{amsmath}
\usepackage{amssymb}
\usepackage{amsthm}
\usepackage{epsfig}
\usepackage{graphicx}
\usepackage{enumerate}
\usepackage{cite}
\usepackage{color}
\usepackage{cancel}
\usepackage{verbatim}

\newtheorem{theorem}{Theorem}[section]
\newtheorem{proposition}[theorem]{Proposition}
\newtheorem{corollary}[theorem]{Corollary}

\newtheorem{definition}[theorem]{Definition}
\newtheorem{remark}[theorem]{Remark}

\title{Complete permutation polynomials from exceptional polynomials}
\date{}
\author{D. Bartoli, M. Giulietti, L. Quoos, and G. Zini}
\begin{document}
\maketitle

\begin{abstract} 
We classify complete permutation monomials of degree $\frac{q^n-1}{q-1}+1$ over the finite field with $q^n$ elements, for $n+1$ a prime and $(n+1)^4 < q$. As a corollary, a conjecture by Wu, Li, Helleseth, and Zhang is proven. When $n+1$ is a power of the characteristic we provide some new examples. Indecomposable exceptional polynomials of degree $8$ and $9$ are also classified.\end{abstract}

{\bf Keywords:} Permutation polynomials, complete permutation polynomials, exceptional polynomials, bent-negabent boolean functions.

\section{Introduction}

Let $\mathbb{F}_{\ell}$ denote the finite field of order $\ell$ and characteristic $p$. A \emph{permutation polynomial} (or PP) $f(x) \in  \mathbb{F}_\ell[x]$ is a bijection of $\mathbb{F}_\ell$ onto itself.
If $f(x)\in\mathbb F_{\ell}$ is a permutation polynomial over $\mathbb F_{\ell^m}$ for infinitely many $m$, then $f(x)$ is said to be an \emph{exceptional polynomial} over $\mathbb F_{\ell}$.
A polynomial $f(x) \in \mathbb{F}_\ell[x]$ is a \emph{complete permutation polynomial} (or CPP) of $\mathbb F_{\ell}$ if both $f(x)$ and $f(x) + x$ are permutation polynomials of $\mathbb{F}_\ell$. 

CPPs are also related to bent and negabent functions which are studied for a number of applications in cryptography, combinatorial designs, and coding theory; see for instance \cite{SGCG,MN,C,ZQ2015}.
% in particular, some characterization results link  negabent functions to particular classes of permutation polynomials over finite fields of even characteristic; see \cite{ZQ2015}.

%\textcolor{blue}{MA VERAMENTE QUESTO CAPOVERSO VI SEMBRA PIAZZATO BENE QUI? BASTA E AVANZA DIRE CHE UN IMPORTANTE CLASSE DI PP SONO GLI EXCEPTIONAL POLYNOMIALS, E DARE LA DEF. LA DISCUSSIONE CON LA CURVA VA MESSA NELLE SEZIONI SUCCESSIVE QUANDO SERVE}
%Let $\mathcal C_f$ be the plane algebraic curve with equation $\frac{f(x)-f(y)}{x-y}=0$ associated to a polynomial $f(x) \in  \mathbb{F}_\ell [x]$. It is easily seen that if $f$ is a PP of $\mathbb{F_\ell}$, then $\mathcal{C}_f$  cannot have an affine point with coordinates in $\mathbb F_\ell$ off the line $x=y$. On the other hand, if $\mathcal C_f$ has no absolutely irreducible component defined over $\mathbb F_\ell$, with the only possible exception of the line $x=y$, then $f(x)$ is a permutation polynomial over an infinite number of extensions of $\mathbb F_\ell$; see also \textcolor{blue}{PERCHE' 'ALSO'?? CHE ALTRO AVETE CITATO? VOLEVATE DIRE 'FOR INSTANCE'?} \cite{C1970}. In this case, $f(x)$ is said to be an {\em exceptional polynomial} over $\mathbb F_\ell$.

%\textcolor{red}{bla bla cancellato}

%In general, permutation polynomials  over finite fields without any prescribed structure are not difficult to construct. Specific additional properties are required by applications in different fields of mathematics and engineering. 

The most studied class of CPPs is the monomial one. If there exists a complete permutation monomial of degree $d$ over $\mathbb F_\ell$, then $d$ is called a CPP exponent over $\mathbb F_\ell$. Complete permutation monomials have been investigated in a number of recent papers, especially for $\ell=q^n$ and  $d=\frac{\ell-1}{q-1}+1$.  Note that  for an element $\alpha\in \mathbb{F}_\ell^*$, the monomial $\alpha x^d$  is a CPP of $\mathbb F_\ell$ if and only if $\gcd(d, \ell -1) = 1$ and $\alpha x^d +x$ is a PP of $\mathbb F_\ell$. 
In \cite{BZ,BZ2015,WLHZ,WLHZ2015} PPs of type $f_b(x)=x^{\frac{q^n-1}{q-1}+1} +bx$ over $\mathbb{F}_{q^n}$ are thoroughly investigated  for $n=2$, $n=3$, and $n=4$. For $n=6$,
 sufficient conditions  for $f_b$ to be a PP of $\mathbb{F}_{q^6}$ are provided in \cite{WLHZ,WLHZ2015} in the special cases of  characteristic  $p\in\{2,3,5\}$, whereas in \cite{BGZ} all $a$'s for which $ax^{\frac{q^6-1}{q-1}+1}$ is a CPP over $\mathbb{F}_{q^6}$ are explicitly listed. The case $p=n+1$ is dealt with in \cite{MZFG}. 

In this paper we discuss monomials of degree $d=\frac{q^n-1}{q-1}+1$ for general $n$. The starting point of our investigation is the observation that $b^{-1}x^d \in \mathbb F_{q^n}[x]$ is a CPP of $\mathbb F_{q^n}$ if and only if $b,b^q,\ldots,b^{q^{n-1}}$ are the roots of 
$$
v_g(x)=\frac{g(-x)-g(0)}{-x}  \in \mathbb F_q[x]
$$
for some permutation polynomial $g(x)$ of degree $n+1$ over $\mathbb F_q$ such that the first-degree term is not zero. If for a root $b$ of $v_g(x)$ the monomial $b^{-1}x^d$ is a CPP over $\mathbb F_{q^n}$, then $g(x)$ will be called a {\em good} PP over $\mathbb F_q$; in this case, all roots of $v_g(x)$ have the same property. Clearly, a PP $g(x)$ over $\mathbb F_q$ is good if and only if the roots of $v_g(x)$ in the algebraic closure of $\mathbb F_q$ form a unique orbit under the action of the Frobenius map $x\mapsto x^q$. 

Our aim is to classify good permutation polynomials over $\mathbb F_{q} $. Here we achieve this goal for all $n$, $(n+1)^4 <q$, with the exception of the cases $n+1=p^r$, with $r>1$, and $n+1=p^r(p^r-1)/2$, with $p\in\{2,3\}$. For $n+1=p^{r}$ we provide several examples. 
Proposition \ref{dickson} shows that, if $q=p^k$ and $n+1$ is a prime different from $p$ satisfying $\gcd(n,k)=\gcd(n+1,p^2-1)=1$, then $d=\frac{q^n-1}{q-1}+1$ is a CPP exponent over $\mathbb F_{q^n}$.
This solves a conjecture by Wu, Li, Helleseth, and Zhang, see \cite[Conjecture 4.18 and Proposition 4.19]{WLHZ2015}.
Our classification implies a result by Bhattacharya and Sarkar (see \cite[Theorem 1.1]{BS}) which determines the PPs of type $f_b$ when $p=2$, $n$ is a power of $2$, and $b\in\mathbb{F}_{q^2}$.

%\textcolor{red}{\cancel{In this way we solve} Proposition \ref{dickson} proves a conjecture by Wu, Li, Helleseth, and Zhang \cite[Conjecture 4.18]{WLHZ2015}, stating that, if $q=p^k$ and $n+1$ is a prime different from $p$ satisfying $\gcd(n,k)=\gcd(n+1,p^2-1)=1$, then there exists $b\in\mathbb F_{q^n}^*$ such that $g(x)=x\prod_{i=0}^{n-1}\left(x+b^{q^i}\right)$ is a Dickson polynomial over $\mathbb F_q$.
% By \cite[8.4.11]{MP2013}, this is equivalent to require that $g(x)$ is a non-monomial exceptional polynomial.}

% \textcolor{blue}{POTRESTE ENUNCIARE LA CONGETTURA?}
Note that since every permutation polynomial with degree less than $q^{1/4}$ is exceptional (see \cite[Theorem 8.4.19]{MP2013}), condition $(n+1)^4<q$ allows us to consider only exceptional polynomials. 
A key tool in our investigation is the classification of indecomposable exceptional polynomials of degree different from $p^r$ for some $r>1$; see \cite[Section 8.4]{MP2013}.

%we in fact obtain a characterization of all monomial CPPs of type $ax^d$,  provided that $(n+1)^4<q$, %$n+1$ is not a power of the characteristic with degree larger than one, and $n+1\neq p^r(p^r-1)/2$ %with $p\in\{2,3\}$. 

If $g(x)$ is a good PP over $\mathbb F_q$ then it is easily seen that $c\cdot g(c'x)+e$ is a good PP over $\mathbb F_q$ for each $c,c',e\in \mathbb F_q$ with $cc'\neq 0$. In this paper two PPs $g(x)$ and $h(x)$ over $\mathbb F_q$ will be called \emph{CPP-equivalent} if there exist $c,c',e\in \mathbb F_q$ with $cc'\neq 0$ such that $h(x)=c\cdot g(c'x)+e$. Note that for $g(x)$ a PP over $\mathbb F_q$ and $k\in\mathbb F_q$, the permutation polynomials $g(x+k)$ and $g(x)$ are equivalent in the usual sense but not CPP-equivalent; in fact, it's possible that one of them is good but the other is not.
Note that, when $g'(x)$ ranges over the CPP-equivalence class of $g(x)$, the roots of $v_{g'}(x)$ range over the roots of $v_g(x)$ and their multiples by non-zero elements in $\mathbb F_q$.
We will consider only one polynomial in a CPP-equivalence class. In particular, we  assume that $g(x)$ is monic and that $g(0)=0$. Since exceptional polynomials only exist for degrees coprime with $q-1$, when  $n$ is odd  we assume that $p=2$.

Our first result is that if $g$ is decomposable, that is $g$ is a composition of two exceptional polynomials with degree grater than one, then $g$ is not good; see Proposition \ref{Indecomponibili}. 

If $g(x) \in \mathbb{F}_{q}[x]$ is a monic indecomposable exceptional polynomial of degree $n+1$ with $g(0)=0$, then, up to CPP-equivalence, one of the following holds \cite[Section 8.4]{MP2013}.
\begin{itemize}
\item[A)] $n+1$ is a prime different from $p$ not dividing $q-1$, and
\begin{enumerate}
\item[A1)] $g(x)=(x+e)^{n+1}-e^{n+1}$, with $e\in \mathbb F_q$, or
\item[A2)] $g(x)=D_{n+1}(x+e,a)-D_{n+1}(e,a)$,  where $a,e\in \mathbb F_q$, $a\ne0$, $n+1\nmid q^2-1$, and $D_{n+1}(x,a)$ denotes a Dickson polynomial of degree $n+1$.
\end{enumerate}

\item[B)] 
 $n+1=p$ and $g(x)=(x+e) ((x+e)^{\frac{p-1}{r}}-a)^r-e(e^{\frac{p-1}{r}}-a)^r$, with $r\mid p-1$, $a,e\in \mathbb F_q$, and $a^{r(q-1)/(p-1)}\neq 1$.

\item[C)] $n+1=s(s-1)/2$, 
where $p\in \{2,3\}$, $q=p^m$,  $s=p^r>3$, and $(r,2m)=1$.

\item[D)] $n+1=p^r$ with $r>1$.
\end{itemize}

For the case $n+1=p^r$, $r>1$, Guralnick and Zieve conjectured in \cite{GZ} that there are no examples of indecomposable exceptional polynomials other than those described in \cite[Propositions 8.4.15, 8.4.16, 8.4.17]{MP2013}.
% \textcolor{blue}{CAMBIARE LA  FRASE SUCCESSIVA COME SEGUE: It has been conjectured by Guralnick and Zieve REFERENZA PRECISA; MA DEVE ESSERE PRECISA, DEVONO SCRIVERE 'WE CONJECTURE' that for the case $n+1=p^r$, $r>1$, there are no examples of indecomposable exceptional polynomials other than those described in \cite[Propositions 8.4.15, 8.4.16, 8.4.17]{MP2013}. }
%\textcolor{red}{Non si trova questa congettura così esplicita: Guralnick e Zieve in POLYNOMIALS  WITH PSL(2) MONODROMY dicono 'we suspect that no further examples exist perch\'e ecc'. poi Zieve lo riprende nell'handbook dicendo 'it is expected...'}
%Also, as \textcolor{red}{\cancel{mentioned} argued} \textcolor{red}{by Guralnick and Zieve in \cite{GZ} and} \cite[Remark 8.4.18]{MP2013}, it is expected that for the case $n+1=p^r$, $r>1$, there are no examples of indecomposable exceptional polynomials other than those described in \cite[Propositions 8.4.15, 8.4.16, 8.4.17]{MP2013}.
% \cite[Proposition 8.4.15]{MP2013}, \cite[Proposition 8.4.16]{MP2013}, \cite[Proposition 8.4.17]{MP2013}.

The paper is organized as follows.
We classify good exceptional polynomials of type A) and B) in Sections \ref{section:A} and \ref{section:B}; see Theorems \ref{typeA} and \ref{Caso_primo}.
We show in Section \ref{section:C} that certain exceptional polynomials of type C) are not good; see Proposition \ref{typeC}. We describe in Section \ref{section:D} some good exceptional polynomials of type D); see Propositions \ref{linearizzati}, \ref{lin2}, and \ref{lin3}.
Finally, we  determine all the exceptional polynomials of degree $8$ and $9$ (see Propositions \ref{$n=7$} and \ref{exceptional}); in this way we provide a proof of the above mentioned Guralnick-Zieve conjecture for the special cases $n=7,8$. As a byproduct, we obtain all the CPPs with $n+1=8$ and $n+1=9$; see Corollaries \ref{coro1} and \ref{coro2} in Section \ref{section:89}.

\section{Preliminaries}

Throughout the paper, $q$ will be a power $p^m$ of a prime $p$ and $\zeta_s$ will denote a $s$-th primitive root of unity, for $s \geq 1$. We begin this section by rephrasing a result by Wu, Li, Helleseth, and Zhang \cite{WLHZ2013}.
For $b\in \mathbb F_{q^n}$, let
$
A_i(b)\in \mathbb F_q
$
denote the evaluation of the $i$-th elementary symmetrical polynomial in $b,b^q,\ldots,b^{q^{n-1}}$, that is,
$$A_i(b)=\sum_{0\le j_1<j_2<\ldots<j_i\le n-1}b^{q^{j_1}+q^{j_2}+\ldots+q^{j_i}}.$$
As a  matter of notation let $A_0(b)=1$.
Recall that $b,b^q,\ldots,b^{q^{n-1}}$ are the roots of the polynomial
$$
(-1)^nA_n(b)+(-1)^{n-1}A_{n-1}(b)T+\ldots+(-1)^{n-i}A_{n-i}(b)T^{i}+\ldots+T^n.
$$
By  \cite[Lemma 5]{WLHZ2013} we have the following result.
\begin{proposition}\label{CPPexceptional}
Assume that $(n+1)^4<q$. The monomial $b^{-1}x^{\frac{q^n-1}{q-1}+1}$ is a CPP over $\mathbb F_{q^n}$ if and only if 
$\gcd(n+1,q-1)=1$ and $\sum_{i=0}^n A_{n-i}(b)x^{i+1}$ is an exceptional polynomial over  $\mathbb F_{q}$.
\end{proposition}

Let $g(x)=\sum_{i=0}^{n+1}\lambda_{n+1-i}x^i$ be an exceptional polynomial over $\mathbb F_q$, and assume that $\lambda_n\neq 0$ and $\lambda_0=1$. Consider the polynomial 
$$
h_g(x)=\frac{g(x)-g(0)}{x}=\frac{\sum_{i=1}^{n+1} \lambda_{n+1-i}x^i}{x}=\sum_{i=0}^{n}\lambda_{n-i}x^i.
$$
Then $v_g(x):=h_g(-x)=\sum_{i=0}^{n}(-1)^i\lambda_{n-i}x^i$.
Note that if $n$ is even, then $h_g(-x)$ can be written as $\sum_{i=0}^{n}(-1)^{n-i}\lambda_{n-i}x^i$. 
If $n$ is odd, then $p=2$ and the same relation holds.

This means that, for any root $b$ of $v_g(x)$, the monomial $b^{-1}x^{\frac{q^n-1}{q-1}+1}$ is a CPP over $\mathbb F_{q^n}$ if and only if 
the roots of $v_g(x)$, or equivalently $h_g(-x)$, form a unique orbit under the Frobenius map $x\mapsto x^q$. This motivates the following definition.

\begin{definition}
An exceptional polynomial $g(x) \in \mathbb{F}_q[x]$ with $g(0)=0$ and $g'(0) \ne 0$ is said to be good if 
the roots of $\frac{g(-x)}{-x}$ form a unique orbit under the Frobenius map $x \mapsto x^q$.
\end{definition}

Therefore, the following has been proved.

\begin{proposition} Assume that $(n+1)^4<q$. Then the elements $b\in\mathbb F_{q^n}\setminus \mathbb F_q$ such that $b^{-1}x^{\frac{q^n-1}{q-1}+1}$ is a CPP over $\mathbb F_{q^n}$ are the roots of polynomials $\frac{g(-x)}{-x}$, for $g$ ranging over good exceptional polynomials of degree $n+1$ over $\mathbb F_q$, with $g(0)=0$ and $g'(0) \ne 0$. 
\end{proposition}

Note that $h_g(x)$ can be viewed as the bivariate polynomial $\frac{g(x)-g(y)}{x-y}$ evaluated at $y=0$. So, assume that we know the factorization of $\frac{g(x)-g(y)}{x-y}$ into absolutely irreducible factors defined over the algebraic closure of $\mathbb F_q$, say
 $\frac{g(x)-g(y)}{x-y}=\prod_{k=1}^s \ell_k(x,y)$. Then
 $$
 h_g(x)=\prod_{k=1}^s \ell_k(x,0).
 $$ 
Obviously, this can be extremely useful to establish whether an exceptional polynomial $g$ is good or not.
Recall that an exceptional polynomial $g(t)$ is decomposable if there exist exceptional polynomials  $g_1$, $g_2$ with degree greater than $1$ such that $g(x)=g_1(g_2(x))$.  
 
\begin{proposition}\label{Indecomponibili}
If $g(x)$ is a good exceptional polynomial, then $g(x)$ is not decomposable.
\end{proposition}
\begin{proof}
Suppose that $g(x)$ is decomposable and write $g(x)=g_1  (g_2(x))$, with polynomials $g_1,g_2$  such that $\deg(g_1),\deg(g_2)>1$. Then 
$$v_g(x)=\frac{g_1(g_2(-x))-g_1(g_2(0))}{-x}=\frac{g_2(-x)-g_2(0)}{-x}\lambda(g_2(-x)),
$$
with
$$
\lambda(g_2(-x)) = \prod_{i=1}^{\deg(g_1)-1}(g_2(-x)-\beta_i)
$$
for some $\beta_i\in\overline{\mathbb F}_q$.
Since $\frac{g_2(-x)-g_2(0)}{-x}$ is a factor of positive degree defined over $\mathbb{F}_q$, the only possibility for the roots of $v_g(x)$ to form a unique orbit under the Frobenius map is that $v_g(x)$ is a power of $\frac{g_2(-x)-g_2(0)}{-x}$. Note that $0$ cannot be a root of $v_g(x)$, since for $b=0$ the monomial $bx^{\frac{q^n-1}{q-1}+1}$ is not a CPP. On the other hand, any root of a factor $g_2(-x)-\beta_i$ must be a root of $g_2(-x)-g_2(0)$, that is $\beta_i=g_2(0)$.  Therefore, 
$$v_g(x)= \left(\frac{g_2(-x)-g_2(0)}{-x}\right)^{\deg(g_1)} (-x)^{\deg(g_1)-1},$$
which is impossible since $\deg(g_1)>1$.
\end{proof}

\section{CPPs from exceptional polynomials of type A)}\label{section:A}

Throughout this section we assume that $n+1\geq3$ is a prime different from $p$.
We denote by $T_{q^{n/2}}$ the absolute trace map $\mathbb{F}_{q^{n/2}}\to\mathbb{F}_2$, $x\mapsto x+x^2+x^4+\ldots+x^{(q^{n/2})/2}$.
We are going to prove the following result.
\begin{theorem}\label{typeA} Assume that $(n+1)^4<q$. For $i\in\{1,\ldots,n/2\}$ let
$$
\alpha_i=\zeta_{n+1}^i+\zeta_{n+1}^{-i} \text{ and } \beta_{i}=\zeta_{n+1}^i-\zeta_{n+1}^{-i}.
$$
Then the monomial $b^{-1}x^{\frac{q^n-1}{q-1}+1}$ is a CPP of $\mathbb F_{q^n}$ precisely in the following cases.
\begin{itemize}
\item If $p\ne2$:
\begin{enumerate}[i)]
\item the order of $q$ modulo $n+1$ is $n$ and, up to multiplication by a non-zero element in $\mathbb F_q$, $b$ is as follows:
\begin{enumerate}
\item $b=\zeta_{n+1}^i-1$, for some $i \in \{1, \dots , n\}$;
\item for $n/2$ even, $b=e(\alpha_i-2)\pm\sqrt{\beta_i^2(e^2-4a)}$ for some $i \in \{1,\ldots n/2\}$, $a\in \mathbb F_q^*$, and $e \in \mathbb F_q$; 
\item for $n/2$ odd, $b=e(\alpha_i-2)\pm\sqrt{\beta_i^2(e^2-4a)}$ for some $i \in \{1,\ldots n/2\}$, $a\in \mathbb F_q^*$, and $e\in\mathbb F_q$ such that $e^2-4a$ is  a square in $\mathbb F_q$.
\end{enumerate}
\item the order of $q$ modulo $n+1$ is $n/2$, $n$ is not divisible by $4$, 
 and, up to multiplication by a non-zero element in $\mathbb F_q$, $b=e(\alpha_i-2)\pm\sqrt{\beta_i^2(e^2-4a)}$ for some $i \in \{1,\ldots n/2\}$, $a\in \mathbb F_q^*$, and $e \in \mathbb F_q$ such that $e^2-4a$ is $0$ or a non-square in $\mathbb F_q$.
\end{enumerate}
\item
If $p=2$:
\begin{enumerate}[i)]
\item the order of $q$ modulo $n+1$ is $n$ and, up to multiplication by a non-zero element in $\mathbb F_q$, $b=\zeta_{n+1}^i-1$ for some $i\in\{1,\ldots,n\}$;
\item the order of $q$ modulo $n+1$ is $n$ or $n/2$, and
$$ b= z_i:= \varepsilon\delta_i^{2} + (\varepsilon+\varepsilon^2)\delta_i^{4} + \ldots + (\varepsilon+\varepsilon^2+\ldots+\varepsilon^{q^n/4})\delta_i^{q^n/2} \quad \textrm{or} \quad b=z_i+1, $$
where $\varepsilon\in\mathbb{F}_{q^n}$ satisfies $T_{q^{n/2}}(\varepsilon)=1$ and, for some $i\in\{1,\ldots,n\}$, $\delta_i=\frac{1}{\alpha_i}+\frac{a}{e^2}$ and $T_{q^{n/2}}(\delta_i)=1$.
\end{enumerate}
\end{itemize}
\end{theorem}

By Propositions \ref{CPPexceptional} and \ref{Indecomponibili}, the determination of the CPPs of type $b^{-1}x^{\frac{q^n-1}{q-1}+1}$ over $\mathbb F_{q^n}$ relies on the classification of indecomposable exceptional polynomials, which is given in \cite[Section 8.4]{MP2013}. In particular, by \cite[Theorem 8.4.11]{MP2013}, Theorem \ref{typeA} is implied by the results of Sections \ref{typeA1} and \ref{typeA2}.

\subsection{CPPs from exceptional polynomials of type A1)}\label{typeA1}
Throughout this subsection we also assume that $n+1$ does not divide $q-1$. Note that for each $e\neq 0$ the polynomial $g(x)=(x+e)^{n+1}-e^{n+1}$ has a non-zero term of degree one. Also, the $n$ distinct roots of
$
h_g(-x)=\frac{(-x+e)^{n+1}-e^{n+1}}{-x}
$
are
$$
-e(\zeta_{n+1}^i-1),\quad i=1,\ldots,n.
$$
\begin{proposition} Assume that $e \in \mathbb F_q^*$. The polynomial $(x+e)^{n+1}-e^{n+1}$ is a good exceptional polynomial over $\mathbb F_q$ if and only if the order of $q$ modulo $n+1$ is equal to $n$.
\end{proposition}
\begin{proof}
The roots of $h_g(-x)$ form a unique orbit under the Frobenius map if and only if $\zeta_{n+1}$ does not belong to any proper subfield of $\mathbb F_{q^n}$. This is equivalent to the order of $q$ modulo $n+1$ being equal to $n$. 
\end{proof}

\begin{corollary} 
Assume that the order of $q$ modulo $n+1$ is equal to $n$.
Then for $b=e(\zeta_{n+1}^i-1)$ the monomial $b^{-1}x^{\frac{q^n-1}{q-1}+1}$ is a CPP of $\mathbb F_{q^n}$, for each $e\in \mathbb F_q^*$ and $i\in\{1,\ldots,n\}$.
\end{corollary}

\subsection{CPPs from exceptional polynomials of type A2)}\label{typeA2}
Throughout this subsection we further assume that $n+1$ does not divide $q^2-1$.
We begin by considering Dickson polynomials $D_{n+1}(x,a)\in\mathbb F_q[x]$.
Recall that 
$$
D_{n+1}(x,a)=\sum_{k=0}^{n/2}\frac{n+1}{n+1-k}\binom{n+1-k}{k}(-a)^kx^{n+1-2k}\,.
$$

Note that $D_{n+1}(x,a)$ has a non-zero term of degree $1$, for each $a\neq 0$. In \cite[Theorems 7 and 8]{BZ1999} Bhargava and Zieve  provide the factorization of $\frac{D_{n+1}(x+e,a)-D_{n+1}(y+e,a)}{x-y}$, $e\in \mathbb{F}_q$. %In particular, for the Dickson polynomial $g(x)=D_{n+1}(x+e,a)$ we have
%$$
%h_g(-x)=\prod_{i=1}^{n/2}\left((-x+e)^2-\alpha_ie(-x+e)+e^2+\beta_i^2a\right)\, ,
%$$ 
%where $\alpha_i=\zeta_{n+1}^i+\zeta_{n+1}^{-i}$ and $\beta_i=\zeta_{n+1}^i-\zeta_{n+1}^{-i}$.

\begin{proposition}\label{dickson}
The polynomial $g(x)=D_{n+1}(x+e,a)-D_{n+1}(e,a)$, with $a,e\in\mathbb F_q$, $a\ne0$, and $D'_{n+1}(e,a)\ne 0$, is a good exceptional polynomial over $\mathbb{F}_{q}$ if and only if one of the following cases occurs:
\begin{enumerate}[i)]
\item $p\ne2$, $n/2$ is even and the order of $q$ modulo $n+1$ is $n$;
\item $p\ne2$, $n/2$ is odd and either $e^2-4a$ is $0$ or a non-square in $\mathbb{F}_q$ and the order of $q$ modulo $n+1$ is $n/2$, or  $e^2-4a$ is a square in $\mathbb{F}_q$ and the order of $q$ modulo $n+1$ is $n$;
\item 
$p=2$, the order of $q$ modulo $n+1$ is $n$ or $n/2$, and $T_{q^{n/2}}(\delta_1)=1$, where $\delta_i=\frac{1}{\alpha_i}+\frac{a}{e^2}$.
\end{enumerate}
In Cases \it{i)} and \it{ii)}, the roots of $h_{g}(-x)$ are
$$b=-\frac{1}{2}\cdot\left(e(\alpha_i-2)\pm\sqrt{\beta_i^2(e^2-4a)}\right).$$
In Case \it{iii)}, let $\varepsilon\in\mathbb{F}_{q^n}$ with $T_{q^{n/2}}(\varepsilon)=1$. Then the roots of $h_g(-x)$ are
$$ b=\varepsilon\delta_i^{2} + (\varepsilon+\varepsilon^2)\delta_i^{4} + \ldots + (\varepsilon+\varepsilon^2+\ldots+\varepsilon^{q^n/4})\delta_i^{q^n/2} \quad \textrm{and} \quad b+1.$$
\end{proposition}
\begin{proof}
By \cite[Theorem 7]{BZ1999} we have
$$
D_{n+1}(x+e,a)-D_{n+1}(y+e,a)=(x-y)\prod_{i=1}^{n/2}\left((x+e)^2-\alpha_i(x+e)(y+e)+(y+e)^2+\beta_i^2a\right),
$$
where $\alpha_i=\zeta_{n+1}^i+\zeta_{n+1}^{-i}$ and $\beta_i=\zeta_{n+1}^i-\zeta_{n+1}^{-i}$. Then
$$h_{g}(-x)=\frac{D_{n+1}(-x+e,a)-D_{n+1}(e,a)}{-x}=\prod_{i=1}^{n/2}\left((-x+e)^2-\alpha_ie(-x+e)+e^2+\beta_i^2a\right),$$
that is, since $\alpha_i^2=\beta_i^2+4$,
$$h_{g}(-x)=\prod_{i=1}^{n/2}\left(x^2+xe(\alpha_i-2)  +(\alpha_i-2)((\alpha_i+2)a-e^2)  \right).$$

Note that the values $e(\alpha_i-2)$ are pairwise distinct for $i=1,\ldots,n$; hence, the sets of roots of two distinct quadratic factors of $h_{g}(-x)$ are disjoint. 

Assume $p\ne2$.
Since $\alpha_i^2=\beta_i^2+4$, the roots of $h_{g}(-x)$ are 
$$
-\frac{1}{2}\cdot\left(e(\alpha_i-2)\pm\sqrt{\beta_i^2(e^2-4a)}\right).$$

Since $(\beta_i^2(e^2-4a))^{q^j}=\left(\beta_{iq^j \pmod {n+1}}\right)^2(e^2-4a)$, if the roots of $h_{g}(-x)$ form a unique orbit under the Frobenius map then the order  $ord_{n+1}(q)$ of $q$ in $\mathbb{Z}_{n+1}^*$ must be either $n$ or $n/2$. %We therefore assume that such a condition holds.

Thus, we check when $\beta_i^2(e^2-4a)$ is a non-square in $\mathbb F_{q^{n/2}}$, so that the $(n/2)$-th power of the Frobenius map permutes the roots of $h_{g}(-x)$.
Note that if $ord_{n+1}(q)=n/2$, then $\beta_i^{q^{n/2}}=\beta_i$ and therefore $\beta_i^2$ is a square in $\mathbb{F}_{q^{n/2}}$; if on the contrary  $ord_{n+1}(q)=n$, then $\beta_i^{q^{n/2}}=-\beta_i$ and $\beta_i^2$ is a non-square in $\mathbb{F}_{q^{n/2}}$. 
Also, $n/2$ even implies that $(e^2-4a)$ is always a square in $\mathbb{F}_{q^{n/2}}$, whereas if $n/2$ is odd then $(e^2-4a)$ is a square in  $\mathbb{F}_{q^{n/2}}$ if and only if it is a square in $\mathbb F_q$.

If $e^2-4a=0$, then $h_g(-x)$ is a square and its roots form a unique orbit under Frobenius. This completes the proof for $p\ne2$.

For $p=2$, similar computations using the solutions of quadratic equations in characteristic $2$ provide the claim.
\end{proof}

\section{CPPs from exceptional polynomials of type B)}\label{section:B}

Throughout this section we assume that $n+1=p$.
For $p=2$, it is straightforward that there exist no exceptional polynomials of type B); hence, we assume that $p\ne2$.
We denote by $\mathbb{N}_{\mathbb{F}_q/\mathbb{F}_p}$ the norm map $\mathbb{F}_q\rightarrow\mathbb{F}_p$, $x\mapsto x^{1+p+p^2+\cdots+q/p}$.

\begin{theorem}\label{Caso_primo} 
Assume that $(n+1)^4<q$. The monomial $b^{-1}x^{\frac{q^n-1}{q-1}+1}$  is a CPP of $\mathbb F_{q^n}$ if and only if, for some divisor $r$ of $n$, one of the following cases occurs:
\begin{enumerate}[i)]
\item $b$ is an element of $\left\{-\zeta_r^i\alpha \ | \ i\in\{0,\ldots,r-1\}, \ \alpha^r=\zeta_{q-1}^j, \ \gcd(r,j)=1\right\}$, or% \cancel{$\{\zeta_{q-1}^i \ | \ \gcd(r,i)=1\}$}
%\item $b$ is an element of $\{\zeta_{q-1}^i \ | \ \gcd(r,i)=1\}$, or
\item $b$ is an element of
$$\Big\{(v_0-\lambda u_0)^\frac{p-1}{r}-e \,\Big|\,  \lambda \in \mathbb F_p^*, \ e,u_0^{p-1} \in \mathbb{F}_q^*, \ u_0^{\frac{(p-1)(q-1)}{r}}\neq 1,$$
$$\hspace{1.8 cm} \ v_0^\frac{p-1}{r}=e,\ ord\left(\mathbb{N}_{\mathbb{F}_q/\mathbb{F}_p}\left(\frac{u_0^{p-1}}{e^{r}}\right)\right)=p-1\Big\}.$$
\end{enumerate}
\end{theorem}
\begin{proof}
%Let $k=n/r$. 
Up to CPP-equivalence, the only indecomposable exceptional polynomials of degree $p$ over $\mathbb F_q$ are the polynomials
$$g(x)=(x+e)\left((x+e)^r-a\right)^{k},$$
where $r$ is a divisor of $n$ and $k=n/r$, 
with $a,e \in \mathbb{F}_q$, $a^{\frac{q-1}{r}}\neq 1$; see \cite[Theorem 8.4.14]{MP2013}.
Hence, 
$$
h_g(-x)=\frac{1}{-x}\Big((-x+e)\left((-x+e)^r-a\right)^{k}-e\left(e^r-a\right)^k\Big).
$$
We distinguish a number of cases.
\begin{itemize}
\item $a=0$. In this case the polynomial $g(x)=(x+e)^p$ is not good.
\item $e=0$ and $a\ne0$. We have that  $h_g(-x)=((-x)^r-a)^k$ has $r$ distinct roots with multiplicity $k$, namely $-\zeta_r^i\alpha$, where $\alpha^r=a$ and $i=0, \dots, k-1$. They form a single orbit under the Frobenius map if and only if $x^r-a$ is irreducible over $\mathbb F_q$. By \cite[Theorem 3.75]{LN}, this is equivalent to require that $a=\zeta_{q-1}^j$ with $\gcd(r,j)=1$.

\item $e\neq0$ and $a\ne0$.
Fix $u_0$, $v_0$ such that $u_0^{p-1}=a$ and $v_0^k=e$. 
It is straightforward to check that the set of roots of $h_g(-x)$ contains  
$$
R=\left\{(v_0-\lambda u_0)^k-e\mid \lambda \in \mathbb F_p^*\right\}.
$$
Note that $e^r\ne a$, since $a^{\frac{q-1}{r}}\neq 1= \left(e^r\right)^{\frac{q-1}{r}}$.
We show that $R$ actually consists of the $p-1$ distinct roots of $h_g(-x)$.
Assume on the contrary that  $(v_0-\lambda u_0)^k-e=(v_0-\lambda' u_0)^k-e$ for some $\lambda\neq \lambda'$. Then  $v_0-\lambda u_0=\mu(v_0-\lambda' u_0)$ for some $\mu$ with $\mu^k=1$, and hence
$
v_0(1-\mu)=u_0(\lambda-\mu \lambda').
$
Since $k$ divides $p-1$, both $\mu$ and $\mu-1$ lies in $\mathbb F_p$. As $\lambda \neq \lambda'$ we have $\mu\neq 1$ and hence
$
1=(v_0/u_0)^{p-1}=e^{\frac{p-1}{k}}/a=e^r/a,
$
a contradiction.

In the following we prove that the elements of $R$ are  in the same orbit under the Frobenius map if and only if 
$$ord\left(\mathbb{N}_{\mathbb{F}_q/\mathbb{F}_p}\left(\frac{a}{e^{r}}\right)\right) =p-1.$$

Let $i\in\{1,\ldots,p-1\}$ be the smallest positive integer such that $((v_0-\lambda u_0)^k-e)^{q^i}=(v_0-\lambda u_0)^k-e$,
so that the elements of $R$ are in the same orbit under the Frobenius map if and only if $i=p-1$.

Since $u_0^{q^i}=u_0a^{(q^i-1)/(p-1)}$ and $v_0^{q^i}=v_0e^{(q^i-1)/k}$, the condition $(v_0-\lambda u_0)^{kq^i}=(v_0-\lambda u_0)^k$ holds if and only if
$$ (v_0e^{(q^i-1)/k}-\lambda u_0a^{(q^i-1)/(p-1)})^{k}=(v_0-\lambda u_0)^k\,, $$
which is equivalent to
$$ \left(v_0-\lambda u_0\frac{a^{(q^i-1)/(p-1)}}{e^{(q^i-1)/k}}\right)^{k}=(v_0-\lambda u_0)^k\,, $$
that is,
$$ \left(v_0-\lambda u_0\frac{a^{(q^i-1)/(p-1)}}{e^{(q^i-1)/k}}\right)=\xi(v_0-\lambda u_0)\,, $$
where $\xi^k=1$.
Suppose $\xi \neq 1$, then 
$$v_0/u_0=\lambda\frac{\frac{a^{(q^i-1)/(p-1)}}{e^{(q^i-1)/k}}-\xi}{1-\xi}\in \mathbb{F}_p^*\,,$$
 and hence $(v_0/u_0)^{p-1}=1$; this implies $a=e^{(p-1)/k}=e^r$, impossible. This means $\xi=1$, that is
\begin{equation}\label{condizione}
\frac{a^{(q^i-1)/(p-1)}}{e^{(q^i-1)/k}}=1\,.
\end{equation}
Since
$$ \frac{q^i-1}{p-1}\equiv \frac{i(q-1)}{p-1}\pmod{q-1} \quad\textrm{ and }\quad \frac{q^i-1}{k}\equiv \frac{i(q-1)}{k}\pmod{q-1}\,,  $$
Equation \eqref{condizione} is equivalent to
$$ \frac{a^{i(q-1)/(p-1)}}{e^{i(q-1)/k}}=1\,, $$
that is,
$$ \left(\mathbb{N}_{\mathbb{F}_q/\mathbb{F}_p}\left(\frac{a}{e^{(p-1)/k}}\right)\right)^i=1\,. $$ 
Therefore, $i=p-1$ if and only if $ord\left(\mathbb{N}_{\mathbb{F}_q/\mathbb{F}_p}\left(\frac{a}{e^{(p-1)/k}}\right)\right)=p-1$. The thesis follows.
\end{itemize}
\end{proof}

\section{CPPs from exceptional polynomials of type C)}\label{section:C}

In this section we deal with one of the three classes of exceptional polynomials of type C), namely the third class in \cite[Theorem 8.4.12 with $e=1$]{MP2013}.

\begin{proposition}\label{typeC}
Let $p=3$,  $s=p^r>3$, $\gcd(r,2m)=1$. The exceptional polynomial
$$
f_e(x)=(x+e)((x+e)^{2}-a)^{(s+1)/4}\left( \frac{((x+e)^{2}-a)^{(s-1)/2}+a^{(s-1)/2}}{(x+e)^{2}}\right)^{(s+1)/2},
$$
where $a$ is a non-square in $\mathbb{F}_q^*$,
%$a \in \mathbb{F}_q^*$ is an element whose image in $\mathbb{F}_q^*/\left(\mathbb{F}_q^*\right)^{2}$ has even order,
is not good over $\mathbb{F}_q$.
\end{proposition}
\begin{proof} Following  \cite[Prop. 2]{Zieve1998},
consider $\tau(y)= (Ey+F)/(\overline{F}y +\overline{E})$, with $E, F,\overline{E},\overline{F} \in \mathbb{F}_{q^2}$ and $E\overline{E}-F\overline{F}=1$. The points $(x,y)$ of the curve with equation $\frac{f_0(x)-f_0(y)}{x-y}$  are exactly the points such that $x=\tau(y)$, where the choice of $(E, F)$ is unique up to replacing $(E,F)$ by $(-E,-F)$ and one of the following cases occurs:
\begin{itemize}
\item $F\overline{F}=-1/2$ ;
\item $\left(EF\overline{E}\overline{F}\right)^{(q-1)/2}=-1$ and $F\overline{F}\neq-1/2$ ;
\item $\left(EF\overline{E}\overline{F}\right)^{(q-1)/2}=1$ .
\end{itemize}

For $e\in \mathbb{F}_q$, a zero  of $h_{f_{e}}(-x)=\frac{f_{e}(-x)-f_e(0)}{-x}$ corresponds to the point $(-x+e,e)$ of the curve defined by $\frac{f_0(x)-f_0(y)}{x-y}=0$, that is $x=y-\tau(y)$ for some $\tau$ as described above.
Since $E, F, \overline{E},\overline{F}\in\mathbb{F}_{q^2}$, we have $\left(\tau(y)\right)^{q^2}=\tau(y)$ and $x^{q^2}=x$. Therefore, the roots of $h_{f_{e}}(-x)$ are not in a unique orbit under the Frobenius map. 
\end{proof}

\section{CPPs from exceptional polynomials of type D)}\label{section:D}

Throughout this section we assume that $n+1=p^r$ with $r>1$.
No complete classification of indecomposable exceptional polynomials of type D) is known. The following propositions deal 
with the cases related to linearized polynomials.

\begin{proposition}\label{linearizzati}
Let $j,k\geq1$ and $H(x)\in\mathbb F_q[x]$ such that $L(x)=x^jH(x^k)$ is a linearized polynomial of degree $n+1$. For $e\in\mathbb F_q$ we have that $S_e(x)=(x+e)^jH^k(x+e)-e^jH^k(e)$ is a good exceptional polynomial over $\mathbb F_q$ if and only if the elements $e-(e_0-\ell)^k$ belong to a unique orbit under the Frobenius map, where $e_0$ is a fixed $k$-th root of $e$ and $\ell$ ranges over the roots of $L(x)\setminus\{0\}$. 
\end{proposition}

\begin{proof}
Following \cite[Theorem 2.1]{Cohen1990} we give the factorization of the curve defined by $S_0(x^k)-S_0(y^k)=0$. Let $N:=\deg(H)=\frac{(n+1)-j}{k}$ and  write
$$H(t)=\prod_{h=1}^{N}(t-\gamma_h),$$
where $\gamma_h \in \overline{\mathbb{F}}_q$. Then the roots of $H(t)$  and $L(x)=x^jH(x^k)$ are  
$\mathcal{H}=\{\gamma_h : h=1,\ldots,N\}$
and
$\mathcal{L}=\{\zeta_k^i \gamma_h : i=0,\ldots,k-1, \ h=1,\ldots,N\}\cup \{0\},$  
respectively. 

Since
$S_0(x^k)=(L(x))^k$, we have
$$S_0(x^k)-S_0(y^k)=\left(L(x)\right)^k-\left(L(y)\right)^k=\prod _{i=0}^{k-1}\left(L(x)-\zeta_k^i L(y)\right)=\prod _{i=0}^{k-1}L(x-\zeta_k^i y)$$
$$=\left(x^k- y^k\right)^j\prod _{\alpha=0}^{d-1}\prod _{\beta=0}^{d-1}\prod _{h=1}^{N}\left(y-\zeta_k^\alpha x -\zeta_k^{\beta} \gamma_h\right).$$

Consider the curve $\mathcal{C}_S$ defined by $S_0(x)-S_0(y)=0$. Clearly, the points $(x,y)$ of $\mathcal{C}_S$ satisfy $\overline{y}=\zeta_k^{\alpha}\overline{x}+\zeta_k^{\beta}\gamma_h$, where $h\in\{1,\ldots,N\}$, $\alpha,\beta\in\{0,\ldots,k-1\}$, $\overline{x}^k=x$, $\overline{y}^k=y$. 

Now consider the polynomial
$h_{S_e}(-x)=\frac{S_e(-x)-S_e(0)}{-x}$. The zeros of $h_{S_e}(-x)$ correspond to the points $(-x+e,e)$ of $\mathcal{C}_S$, $x\neq 0$. Fix $e_0$ such that $e_0^k=e$; then the zeros of $h(-x)$ are 
$\left \{e-(e_0 -\ell)^k\ | \ \ell \in \mathcal{L}\setminus \{0\} \right\}$.
\end{proof}

In general, it is not easy to establish when the elements $e-(e_0-\ell)^k$ belong to the same orbit under the Frobenius map. The following propositions provide two families of good exceptional polynomials arising from linearized polynomials.

\begin{proposition}\label{lin2}
Let $q=p^m$ and $L(x)=x^{p^r}-\zeta_{q-1} x \in\mathbb F_q[x]$.
%, with $\xi$ a primitive element of $\mathbb F_q, q=p^m$.
If $r$ divides $m$, then $L(x)$ is good exceptional over $\mathbb F_q$.
\end{proposition}
\begin{proof}
Let $N=p^r-1$, and let $\eta\in\mathbb F_{q^{N}}$ be a root of $h_L (-x)=x^{N}-\zeta_{q-1}$. Then the roots of 
$h_L (-x)$ are $\{\lambda\eta \mid \lambda\in\mathbb F_{p^r}^*\}$.
The hypothesis $r \mid m$ is equivalent to require that $N$ divides $(q-1)$, and this implies that $N(q-1)\mid q^N-1$. Hence we can choose $\eta=\omega^{\frac{q^N-1}{N(q-1)}}$, where $\omega$ is a primitive element of $\mathbb F_{q^N}$.
The thesis is proved by showing that $\eta$ is not an element of any proper subfield of $\mathbb F_{q^{N}}$.
Suppose that $\eta\in\mathbb F_{q^{k}}$ with $k\mid N$. Then $\omega^{\frac{q^N-1}{N(q-1)}(q^k-1)}=1$, that is $N\mid\frac{q^k-1}{q-1}$; since $q\equiv1\pmod N$, this is equivalent to $N\mid k$, and hence to $N=k$.
\end{proof}

\begin{proposition}\label{lin3}
If $\gcd(m,p^r-1)$ is a divisor of $r$, then there exists a linearized polynomial $L(x)\in\mathbb F_q[x]$ of degree $p^r$ which is good exceptional over $\mathbb F_q$.
\end{proposition}

\begin{proof}
Let $\epsilon=\gcd(m,p^r-1)$ and $\ell(x)\in\mathbb F_q[x]$ be a primitive polynomial of degree $r/\epsilon$ over $\mathbb F_{p^\epsilon}$, so that $\ell(x)$ is irreducible over $\mathbb F_{p^\epsilon}$ and has order $p^r-1$. Let $L(x)\in\mathbb F_q[x]$ be the linearized $p^\epsilon$-associate of $\ell(x)$. Then, by \cite[Theorem 3.63]{LN}, the polynomial $L(x)/x$ is irreducible over $\mathbb F_{p^\epsilon}$. Let $\alpha$ be a non-zero root of $L(x)$. Then the field extension $\mathbb F_{p^\epsilon}(\alpha) : \mathbb F_{p^\epsilon}$ has degree $p^r-1$, while the extension $\mathbb F_q : \mathbb F_{p^\epsilon}$ has degree $m/\epsilon$. The field $\mathbb F_q(\alpha)$ is the compositum of $\mathbb F_q$ and $\mathbb F_{p^\epsilon}(\alpha)$; since $\gcd(m/\epsilon,p^r-1)=1$, we have that $[\mathbb F_q(\alpha):\mathbb F_q]=p^r-1$. Then $L(x)/x=h_L(-x)$ is irreducible over $\mathbb F_q$, and the thesis follows.
\end{proof}

\section{The cases $n+1=8$ and $n+1=9$}\label{section:89}

The aim of this section is to study the cases $n+1=8$ (with $p=2$) and $n+1=9$ (with $p=3$). Since  no complete classification of exceptional polynomials is known, we study the existence of good exceptional polynomials via algebraic curves associated to a PP. In fact,  a polynomial $f(x)$ is a PP of $\mathbb F_q$ if and only if the algebraic curve $\mathcal C_{f}$ of degree $n$ with equation
$$
\frac{f(x)-f(y)}{x-y}=0
$$
has no $\mathbb F_q$-rational points off the ideal line and the line $x=y$. For $q$ large enough with respect to the degree of $f(x)$, by the Hasse-Weil bound this is only possible when $\mathcal C_f$ splits into components not defined over $\mathbb F_q$; see for instance \cite{BGZ}.
On the other hand, if $\mathcal C_f$ has no absolutely irreducible component defined over $\mathbb F_q$, with the only possible exception of the line $x=y$, then $f(x)$ is a permutation polynomial over an infinite number of extensions of $\mathbb F_q$, that is $f(x)$ is exceptional over $\mathbb F_q$; see \cite{C1970} and \cite[Chapter 8.4]{MP2013}.

\subsection{$n+1=8$, $p=2$}

\begin{proposition}\label{$n=7$}
Let $q=2^{m}$, $n+1=8$. The polynomial $f(x)=x^8+\sum_{i=1}^{7} A_{i}x^{8-i}\in\mathbb F_q[x]$ is exceptional over $\mathbb F_q$ if and only if 
$A_1=A_2=A_3=A_5=0$
and the polynomial $g(x)=x^7+A_4x^3+A_6x+A_7$ has no roots in $\mathbb{F}_q^*$. Also $f(x)$ is good exceptional if and only if $g(x)$ is irreducible over $\mathbb{F}_q$.
\end{proposition}
\begin{proof}
 The equation of the curve $\mathcal{C}_f$ reads
\begin{align*}\label{Curva8}
& (x+y)^7+A_1(x^6+x^5y+x^4y^2+x^3y^3+x^2y^4+xy^5+y^6)+A_2(x^5+x^4y+x^3y^2+x^2y^3+xy^4+y^5)\\&+A_3(x^4+x^3y+x^2y^2+xy^3+y^4) 
 +A_4(x+y)^3+A_5(x^2+xy+y^2)+A_6(x+y)+A_7=0.
\end{align*}

Applying the Frobenius automorphism $t \mapsto t^q $ to the factors of $\mathcal{C}_f$ it is easy to conclude that if the curve  $\mathcal{C}_f$ does not have absolutely irreducible components defined over $\mathbb{F}_q$, then the curve contains either two conics and three lines or seven lines. The unique ideal point of $\mathcal{C}_f$ is $(1:1:0)$. A line $\ell$ that is a component of the curve $\mathcal{C}_f$ has equation $\ell : y=x+\alpha$ and 
$$
\left\{\begin{array}{l}
A_1=0\\
A_2\alpha+A_3=0\\
A_2\alpha^3+A_5=0\\
A_3\alpha^2+A_5=0\\
\alpha^7+A_2\alpha^5+A_3\alpha^4+A_4\alpha^3+A_5\alpha^2+A_6\alpha+A_7=0.
\end{array}\right.
$$
If the line $\ell$ is not defined over $\mathbb{F}_q$ then $\alpha\in \overline{\mathbb{F}}_q\setminus \mathbb{F}_q$; this yields $A_2=A_3=A_5=0$, and the last equality becomes $\alpha^7+A_4\alpha^3+A_6\alpha+A_7=0$. It is easily seen that if $A_2=A_3=A_5=0$ then the curve $\mathcal{C}_f$ contains the seven lines $y+x+\alpha_i=0$, $i=1,\ldots,7$, where $\alpha_i^7+A_4\alpha_i^3+A_6\alpha_i+A_7=0$, and therefore $\mathcal C_f$ cannot split in two conics and three lines.

Thus, the only open case occurs when  $\mathcal{C}_f$ splits in seven lines either not defined over $\mathbb{F}_q$ or equal to $x-y=0$.

Therefore $f(x)$ is exceptional if and only if $T^7+A_4T^3+A_6T+A_7$ has no roots in $\mathbb{F}_q$ and it is good exceptional if and only if $T^7+A_4T^3+A_6T+A_7$ is irreducible over $\mathbb{F}_q$ since all the roots must be in the same orbit. 
\end{proof}

\begin{corollary}
Let $q=2^{m}$, $n+1=8$ and suppose $3$ divides $m$. The polynomial $x^8+A_7x$ is the only good exceptional polynomial over $\mathbb F_q$.
\end{corollary}
\begin{proof}
Since $3$ divides $m$ we have that $\zeta_7 \in \mathbb{F}_q$. From Cyclic extensions theory,  $T^7+A_4T^3+A_6T+A_7$ is irreducible over $\mathbb{F}_q$ if and only if $A_4=A_6=0$. The thesis follows from Proposition  \ref{$n=7$}.
\end{proof}

\begin{remark}
The exceptional polynomials of Proposition {\rm \ref{$n=7$}} are linearized, and hence described in {\rm \cite[Prop. 8.4.15]{MP2013}}. Also, Proposition {\rm \ref{$n=7$}}  confirms the conjecture {\rm \cite[Remark 8.4.18]{MP2013}} for the special case $n+1=8$.
\end{remark}

\begin{corollary}\label{coro1}
Assume that $q=2^r>8^4$. The monomial $b^{-1}x^{\frac{q^7-1}{q-1}+1}$  is a CPP of $\mathbb F_{q^8}$ if and only if $b$ is, up to a scalar multiple in $\mathbb{F}_q^*$, a root of some $F(x)=x^7+\alpha x^3+\beta x+\gamma \in \mathbb{F}_q[x]$,  irreducible over $\mathbb{F}_q$. 
\end{corollary}

\subsection{$n+1=9$, $p=3$}

\begin{proposition}\label{exceptional}
Let $q=3^h$.
The polynomial 
$$F(x)=x^9+A_1x^8+A_2x^7+A_3x^6+A_4x^5+A_5x^4+A_6x^3+A_7x^2+A_8x$$
is exceptional over $\mathbb{F}_q$ if and only if one of the following cases occurs.

\begin{enumerate}[i)]
\item 
\begin{equation}\label{pol1} F(x)=x^9+A_3x^6+A_6x^3
\end{equation}
and $T^6+A_3T^3+A_6\in\mathbb F_q[T]$ has no roots in  $\mathbb{F}_q^*$;
\item 
\begin{equation}\label{pol2} F(x)=x^9+A_6x^3+A_8x
\end{equation}
and $T^8+A_6T^2+A_8\in\mathbb F_q[T]$ has no roots in  $\mathbb{F}_q^*$;

\item 
$$F(x)=x^9+A_3x^6+A_4x^5+A_5x^4+\left(A_3^2 + A_3 \frac{A_5^3}{A_4^3} + \frac{A_5^2}{A_4}\right)x^3$$
 \begin{equation}\label{pol3}\qquad\quad+\left(2A_3 A_4  +2 \frac{A_5^3}{A_4^2}\right)x^2+\left(2A_3A_5 + A_4^2  +2 \frac{A_5^4}{A_4^3}\right)x, \end{equation} 
where 
\begin{enumerate}
\item $A_4\neq 0$,
\item the polynomial $T^4+2A_3T+2A_4\in\mathbb F_q[T]$ has no roots in $\mathbb F_{q}$;
\end{enumerate}
\item 
$$F(x)=x^9+A_2x^7+A_3x^6+A_5x^4+\left(A_2^3 + \frac{A_3 A_5}{A_2}\right)x^3+$$
\begin{equation}\label{pol4}
\left(2A_2 A_5  +2\frac{A_3^3}{A_2}\right)x^2+\left(A_2^4+A_3A_5+\frac{A_5^2}{A_2}+\frac{A_3^4}{A_2^2}\right)x,
\end{equation}
where $2A_2$ is not a square in $\mathbb{F}_q$.
\end{enumerate}

%,  the polynomial 
%$$F(x)=x^9+A_1x^8+A_2x^7+A_3x^6+A_4x^5+A_5x^4+A_6x^3+A_7x^2+A_8x$$
%is good exceptional over $\mathbb{F}_q$ if and only if one of the following cases occurs:
%\begin{enumerate}[i)]
%\item $F(x)=x^9+A_6x^3+A_8x $ and $x^8+A_6x^2+A_8$ irreducible over $\mathbb{F}_q$,\\
%\item 
%\begin{align*}\label{pol2}
%F(x)=&x^9+A_3x^6+A_4x^5+A_5x^4+\left(A_3^2 + A_3 \frac{A_5^3}{A_4^3}\right)x^3\\
%&+\left(2A_3 A_4  +2 \frac{A_5^3}{A_4^2}\right)x^2+\left(2A_3A_5 + A_4^2  +2 \frac{A_5^4}{A_4^3}\right)x,
%\end{align*}
%where 
%\begin{itemize}
%\item[1.] $A_4\neq 0$,
%\item[2.] the polynomial $x^8+2A_3x^2+2A_4\in\mathbb F_q[x]$ has no roots in $\mathbb F_{q^4}$;
%\end{itemize}
%\item 
%\begin{align*}
%F(x)=&x^9+A_2x^7+A_3x^6+A_5x^4+\left(A_3^2 + A_3 A_5/A_2\right)x^3\\
%&+\left(2A_2 A_5  +2A_3^3/A_2\right)x^2+\left(A_2^4+A_3A_5+A_5^2/A_2+A_3^4/A_2^2\right)x,
%\end{align*}
%where 
%\begin{itemize}
%\item[1.] $2A_2$ is not a square in $\mathbb{F}_q$,
%\item[2.] $h(-x)=(x^4+2\alpha_1x^3+2A_3/\alpha_1 x^2+2(A_3+2\alpha_1 A_2)x+A_3\alpha_1  + A_3^2/A_2 + 2A_5\alpha_1/A_2 + A_2^2)(x^4+\alpha_1x^3+A_3/\alpha_1 x^2+2(A_3+\alpha_1A_2)x+2A_3\alpha_1  + A_3^2/A_2 + A_5\alpha_1/A_2 + A_2^2)$, where $\alpha_1^2=2A_2$, is irreducible over $\mathbb{F}_q$.
%\end{itemize}
%\end{enumerate}

\end{proposition}

\begin{proof}

The curve $\mathcal{C}_f$ associated to the polynomial $F(x)=x^9+\sum_{i=1}^8 A_{8-i}x^i$ is
$$\sum_{i=0}^{8} A_{8-i}\frac{x^{i+1}-y^{i+1}}{x-y}=0,$$
where $A_i\in \mathbb{F}_q$, $i=1,\ldots,8$, $A_0=1$, and $A_8\neq 0$. 
\begin{itemize}
\item Suppose that $\mathcal{C}_f$ contains a line $\ell$ with equation $\ell: y=x+\alpha$, where $\alpha=0$ or $\alpha\notin \mathbb{F}_q$. Then $\alpha $ satisfies:
$$
\left\{\begin{array}{l}
2A_1=0\\
2\alpha A_1+A_2=0\\
2 \alpha^2 A_1=0\\
\alpha^3A_1+2 \alpha^2 A_2 + A_4=0\\
2 \alpha^4 A_1+2\alpha^3 A_2 + 2 \alpha^2 A_3+\alpha A_4+A_5=0\\
\alpha^5 A_1 + \alpha^2 A_4=0\\
2\alpha^6 A_1 + \alpha^8 A_2 + \alpha^3 A_4 + \alpha^2 A_5 + 2A_7=0\\
\alpha^8 A_0+\alpha^7 A_1+ \alpha^6 A_2 + \alpha^5 A_3 + \alpha^4 A_4+\alpha^3 A_5 +\alpha^2 A_6 +\alpha A_7+ A_8=0
\end{array}\right..
$$
This implies $A_1=A_2=0$. We distinguish two cases.
\begin{itemize}
\item $\alpha=0$. Then $A_4=A_5=A_7=A_8=0$. The curve becomes
\[(x-y)^2((x-y)^6+A_3(x-y)^3+A_6)=0.\] 
We have to require that the polynomial $T^6+A_3T^3+A_6$ has no roots in $\mathbb{F}_q^*$. 
\item $\alpha\neq 0$. Then $A_4=0$. The conditions above read 
$$A_5\alpha^2 + 2A_7=0, \quad \alpha^8 + A_3\alpha^5 + A_5 \alpha^3 + A_6 \alpha^2 + A_7\alpha + A_8=0,\quad 2A_3\alpha^2 + A_5=0.$$
If $A_5=0$ then $A_3=A_7=0$ and   $\alpha^8 + A_6\alpha^2  + A_8 =0$ and the curve splits in $8$ lines. They are not defined over $\mathbb{F}_q$ or equal to $x-y=0$ if and only if the polynomial $T^8 + A_6T^2  + A_8 =0$ has no roots in $\mathbb{F}_q^*$.

If $A_3=0$ then $A_5=A_7=0$ and   $\alpha^8 + A_6\alpha^2  + A_8 =0$, as above.

Suppose now $A_3,A_5\neq 0$. Then $A_5=A_3\alpha^2$, $A_7 =A_5^2/A_3$, and $A_8 = 2A_5A_6/A_3 + 2A_5^4/A_3^4$. Since $\alpha^2=A_5/A_3$, we have that $A_5/A_3$ is not a square in $\mathbb{F}_q$, otherwise the lines $y=x+\xi_1$ and $y=x+\xi_2$, where $\xi_i^2=A_5/A_3$, are $\mathbb{F}_q$-rational lines and the polynomial $F(x)$ is not exceptional. Let $a_3,a_5\in\mathbb F_{q^2}$ be such that $a_3^2=A_3$ and $a_5^2=A_5$. In this case the curve splits in 
$$(a_3x  -a_3y + a_5)(a_3x  -a_3y - a_5)\cdot $$
$$ \cdot(a_3^6 (x-y)^6     + a_3^4 a_5^2 (x-y)^4 +  a_3^8 (x+y)^3 +  a_3^2 a_5^4 (x-y)^2   -a_3^6 a_5^2 (x+y)   + a_3^6 A_6 + a_5^6)=0.$$

Since the sextic is defined over $\mathbb{F}_q$, it must split either in three conics or in two cubics. In the first case it is easily seen that all of them must be fixed by $\psi$. 

If a conic of equation  $(x-y)^2+\alpha(x+y)+\beta=0$ is contained in the sextic then in particular $A_3^2 =\alpha^3$ from which we get $a_3^{32} a_5^{12}=0$, impossible.

Suppose now that the sextic splits in two cubics. 

If they are fixed by $\psi$ then they have equations
$$(x-y)^3+\alpha_1 x^2+\alpha_2 x y+\alpha_1 y^2+\beta_1 x+\beta_1 y+\gamma_1=0$$
and 
$$(x-y)^3+\alpha_3 x^2+\alpha_4 x y+\alpha_3 y^2+\beta_2 x+\beta_2 y+\gamma_2=0.$$
Then in particular $\alpha_4=\alpha_3=-\alpha_1$, $\alpha_2=\alpha_1$. If $\alpha_1=0$ then $a_3^2 = \gamma_1 + \gamma_2$ and $a_3^2 = -\gamma_1 - \gamma_2$, which imply $a_3=0$, impossible. If $\alpha_1\neq 0$ then $\beta_1=\beta_2$ and again $a_3^2 = \gamma_1 + \gamma_2$ and $a_3^2 = -\gamma_1 - \gamma_2$, which imply $a_3=0$, impossible.

If they are switched by $\psi$ then they have equations
$$(x-y)^3+\alpha_1 x^2+\alpha_2 x y+\alpha_3 y^2+\beta_1 x+\beta_2 y+\gamma_1=0$$
and 
$$\lambda((y-x)^3+\alpha_3 x^2+\alpha_2 x y+\alpha_1 y^2+\beta_2 x+\beta_1 y+\gamma_1)=0.$$
Then in particular $\lambda= -a_3^6$, $\alpha_3=\alpha_1$.
If  $\alpha_2=-\alpha_1$, then $a_3^2 \alpha_1^2 + a_3^2 \beta_1 -a_3^2\beta_2 -a_5^2=0$ and $\alpha_1=0$, which implies $a_3=0$, impossible.
If  $\alpha_2=\alpha_1$, then $\alpha_1(\beta_1+\beta_2)=0$. In both the cases $a_3=0$, impossible.

\end{itemize}

\item Suppose that $\mathcal{C}_f$ splits in four absolutely irreducible conics. There are three distinct possibilities, depending on the number of components fixed by $\psi$.
\begin{enumerate}
\item All the conics are fixed by $\psi$. In this case the four conics are defined by
\begin{equation}\label{theconics} \mathcal{C}_i : (x-y)^2+\alpha_i(x+y)+\beta_i=0, \end{equation}
 for $i=1,2,3,4$. This gives immediately $A_1=A_2=0$.

The condition $A_4= 0$ implies $A_5=A_7=0$ and $A_3A_8=0$, that is the polynomial is either of type \eqref{pol1} or \eqref{pol2}. 

Suppose  $A_4\neq 0$. Then, by direct computation, $A_6 = A_3^2 + A_3 A_5^3/A_4^3  + A_5^2/A_4$, $A_7 = 2A_3 A_4  +2 A_5^3/A_4^2$, $A_8 = 2A_3A_5 + A_4^2  +2 A_5^4/A_4^3$; also, the $\alpha_i$'s are roots of $\ell_1(x)=x^4+2A_3x +2 A_4$, and 
$\beta_i= \alpha_i^2+A_5/A_4\alpha_i.$ On the other hand if all these conditions are satisfied then the curve splits in the four conics defined in \eqref{theconics}. Finally, the four conics are not defined over $\mathbb{F}_q$ if and only if the polynomial $T^4+2A_3T +2 A_4$ has no roots in $\mathbb{F}_q$.

\item Two conics are fixed by $\psi$ and two are switched. We can assume 
$$\mathcal{C}_1 : (x-y)^2+\alpha_1 (x+y)+\beta_1=0, \quad \mathcal{C}_2 : (x-y)^2+\alpha_2 (x+y)+\beta_2=0,$$
$$\mathcal{C}_3 : (x-y)^2+\alpha_3 x+\alpha_4 y+\beta_3=0, \quad \mathcal{C}_4 : (x-y)^2+\alpha_4 x+\alpha_3 y+\beta_3=0.$$
By direct compuation, we get immediately $A_1=A_2=A_4=A_5=A_7=0$ and $A_3A_8=0$, and hence $F(x)$ is of type \eqref{pol1} or \eqref{pol2}. 
\item No conic is fixed by $\psi$. We can assume 
$$\mathcal{C}_1 : (x-y)^2+\alpha_1 x+\alpha_2 y+\beta_1=0, \quad \mathcal{C}_2 : (x-y)^2+\alpha_2 x+\alpha_1y+\beta_1=0,$$
$$\mathcal{C}_3 : (x-y)^2+\alpha_3 x+\alpha_4y+\beta_2=0, \quad \mathcal{C}_4 : (x-y)^2+\alpha_4 x+\alpha_3y+\beta_2=0.$$
Also in this case we get $A_1=A_2=A_4=A_5=A_7=0$ and $A_3A_8=0$, and hence $F(x)$ is of type \eqref{pol1} or \eqref{pol2}.
\end{enumerate}
\item Suppose that $\mathcal{C}_f$ splits in two absolutely irreducible quartics $\mathcal Q_1$ and $\mathcal Q_2$. The automorphism $\psi$ can either switch or fix the two components.

In the former case, $\mathcal Q_1$ and $\mathcal Q_2$ have the form
$$\mathcal{Q}_1 : (x-y)^4+\alpha_1x^3+\alpha_2x^2y+\alpha_3xy^2+\alpha_4y^3+\beta_1x^2+\beta_2xy+\beta_3y^2+\gamma_1x+\gamma_2y+\delta=0,$$
$$\mathcal{Q}_2 : (x-y)^4+\alpha_4x^3+\alpha_3x^2y+\alpha_2xy^2+\alpha_1y^3+\beta_3x^2+\beta_2xy+\beta_1y^2+\gamma_2x+\gamma_1y+\delta=0.$$
We get $A_1=A_2=A_3=A_4=A_5=A_7=0$; hence, we have case \eqref{pol2}.

In the latter case, $\mathcal Q_1$ and $\mathcal Q_2$ have the form
$$\mathcal{Q}_1 : (x-y)^4+\alpha_1x^3+\alpha_2x^2y+\alpha_2xy^2+\alpha_1y^3+\beta_1x^2+\beta_2xy+\beta_1y^2+\gamma_1(x+y)+\delta_1=0,$$
$$\mathcal{Q}_2 : (x-y)^4+\alpha_3x^3+\alpha_4x^2y+\alpha_4xy^2+\alpha_3y^3+\beta_3x^2+\beta_4xy+\beta_3y^2+\gamma_2(x+y)+\delta_2=0.$$ 
Since $A_1=0$, we obtain $A_2A_4=0$. 

\begin{itemize}
\item Suppose $A_2=0$ and $A_4\neq 0$. Then $A_6 = A_3^2 + A_3A_5^3/A_4^3  + A_5^2/A_4$, $A_7 = 2 A_3A_4 + 2A_5^3/A_4^2$, $A_8=  2A_3A_5 + A_4^2 +A_5^4/A_4^3$ and we have case \eqref{pol3}.

\item Suppose now $A_2\neq 0$ and $A_4= 0$. Then $A_6 = A_2^3+ A_3A_5/A_2$, $A_7 = 2A_2A_5 + 2A_3^3/A_2$, $A_8= A_2^4 + A_3A_5  + A_5^2/A_2 + A_3^4/A_2^2$. Also, $\alpha_1^2=2A_2$, $\alpha_2=\alpha_3=-\alpha_4=-\alpha_1$,  $\beta_1=-\beta_3=2A_3/\alpha_1$,  $\beta_2=-A_3/\alpha_1  -\alpha_1^2$, $\beta_4=A_3/\alpha_1  -\alpha_1^2$, $\gamma_1=A_3+\alpha_1^3$, $\gamma_2=A_3-\alpha_1^3$, $\delta_1 = A_3\alpha_1  + A_3^2/A_2 + 2A_5\alpha_1/A_2 + 2\alpha_1^6/A_2$, $\delta_2 = -A_3\alpha_1  + A_3^2/A_2 + A_5\alpha_1/A_2 + 2\alpha_1^6/A_2$. Note that  $\alpha_i,\beta_i,\gamma_i,\delta_i$ are not defined over $\mathbb{F}_q$ if and only if $2A_2$ is not a square in $\mathbb{F}_q$. The quartics $\mathcal Q_1$ and $\mathcal Q_2$ read
$$ (x-y)^4+\alpha_1x^3+2\alpha_1x^2y+2\alpha_1xy^2+\alpha_1y^3+2A_3/\alpha_1 x^2+2(A_3/\alpha_1  +\alpha_1^2)xy$$
$$+2A_3/\alpha_1y^2+(A_3+\alpha_1^3)(x+y)+A_3\alpha_1  + A_3^2/A_2 + 2A_5\alpha_1/A_2 + 2\alpha_1^6/A_2=0,$$
$$(x-y)^4+2\alpha_1x^3+\alpha_1x^2y+\alpha_1xy^2+2\alpha_1y^3+A_3/\alpha_1x^2+(A_3/\alpha_1  +2\alpha_1^2)xy$$
$$+A_3/\alpha_1y^2+(A_3+2\alpha_1^3)(x+y)+2A_3\alpha_1  + A_3^2/A_2 + A_5\alpha_1/A_2 + 2\alpha_1^6/A_2=0,$$ 
and  $\mathcal Q_1$ and $\mathcal Q_2$ are switched by the Frobenius map. 

\item Finally, $A_2=A_4= 0$ implies $A_5=A_7=0$ and $A_3A_8=0$. As above, this gives types \eqref{pol1} or \eqref{pol2}.
\end{itemize}
\end{itemize}
\end{proof}

\begin{remark}
By direct computation, the exceptional polynomials of Proposition {\rm \ref{exceptional}} are equivalent to exceptional polynomials described in {\rm \cite[Prop. 8.4.15]{MP2013}}.

In fact, if $F(x)$ satisfies Case $i)$ or $ii)$, then $F(x)$ is a linearized polynomial.

If $F(x)$ satisfies Case $iii)$, then $F(x)=L_1\circ S\circ L_2(x)$, where $L_1(x)$ and $L_2(x)$ are linear, and $S(x)\in\mathbb F_q[x]$ has the form $x(a_2x^4+a_1x+a_0)^2$.

If $F(x)$ satisfies Case $iv)$, then $F(x)=L_1\circ S\circ L_2(x)$, where $L_1(x)$ and $L_2(x)$ are linear, and $S(x)\in\mathbb F_q[x]$ has the form $S(x)=x(a_2x^4+a_1x+a_0)^2$ when $A_2^2 A_5 + A_3^3\ne0$, or $S(x)=x(a_2x^2+a_0)^4$ when $A_2^2 A_5 + A_3^3=0$.

This confirms the conjecture {\rm \cite[Remark 8.4.18]{MP2013}} for the special case $n+1=9$.
\end{remark}

\begin{proposition}\label{goodexceptional}
Let $q=3^h$.
The polynomial 
$$F(x)=x^9+A_1x^8+A_2x^7+A_3x^6+A_4x^5+A_5x^4+A_6x^3+A_7x^2+A_8x$$
is good exceptional over $\mathbb{F}_q$ if and only if one of the following cases occurs.
\begin{itemize}
\item[i)] 
$$ F(x)=x^9+A_6x^3+A_8x $$
and $x^8+A_6x^2+A_8$ is irreducible over $\mathbb{F}_q$;
\item[ii)] 
$$F(x)=x^9+A_3x^6+A_4x^5+A_5x^4+\left(A_3^2 + A_3 \frac{A_5^3}{A_4^3} + \frac{A_5^2}{A_4}\right)x^3$$
$$\qquad\quad+\left(2A_3 A_4  +2 \frac{A_5^3}{A_4^2}\right)x^2+\left(2A_3A_5 + A_4^2  +2 \frac{A_5^4}{A_4^3}\right)x, $$
where 
\begin{itemize}
\item[(a)] $A_4\neq 0$,
\item[(b)] the polynomial $x^8+2A_3x^2+2A_4\in\mathbb F_q[x]$ has no roots in $\mathbb F_{q^4}$;
\end{itemize}
\item [iii)]
$$F(x)=x^9+A_2x^7+A_3x^6+A_5x^4+\left(A_2^3 + \frac{A_3 A_5}{A_2}\right)x^3+$$
$$
\left(2A_2 A_5  +2\frac{A_3^3}{A_2}\right)x^2+\left(A_2^4+A_3A_5+\frac{A_5^2}{A_2}+\frac{A_3^4}{A_2^2}\right)x,
$$
where 
\begin{itemize}
\item[(a)] $2A_2$ is not a square in $\mathbb{F}_q$,
\item[(b)] $h(-x)=(x^4+2\alpha x^3+2A_3/\alpha x^2+2(A_3+2\alpha A_2)x+A_3\alpha + A_3^2/A_2 + 2A_5\alpha/A_2 + A_2^2)(x^4+\alpha x^3+A_3/\alpha x^2+2(A_3+\alpha A_2)x+2A_3\alpha + A_3^2/A_2 + A_5\alpha/A_2 + A_2^2)$, where $\alpha^2=2A_2$, is irreducible over $\mathbb{F}_q$.
\end{itemize}
\end{itemize}
\end{proposition}

\begin{proof}
We use the classification of exceptional polynomials of degree $9$ given in  Proposition \ref{exceptional}. % and rely on the computations in that proof.
\begin{itemize}
\item Let $F(x)$ be as in Case $i)$ of Proposition \ref{exceptional}.
Then $A_8=0$; hence, $F(x)$ is not good.

\item Let $F(x)$ be as in Case $ii)$ of Proposition \ref{exceptional}.
Then $h_F(-x)=x^8+A_6x^2+A_8$; since $h_F(-x)$ cannot be a square in $\mathbb F_q[x]$, we have that $F(x)$ is good if and only if $h_F(-x)$ is irreducible over $\mathbb F_q$.

\item Let $F(x)$ be as in Case $iii)$ of Proposition \ref{exceptional}.
The factors of $h_F(-x)$ are $x^2-\alpha_i x+\beta_i$, $i=1,\ldots,4$, where the $\alpha_i$'s are roots of $\ell_1(x)=x^4+2A_3x+2A_4$ and $\beta_i=\alpha_i^2+A_5/A_4\alpha_i$; hence, $\ell_1(x)$ must be irreducible over $\mathbb F_q$ in order for $F(x)$ to be good.
Also, the roots of $h(-x)$ are $-\alpha_i\pm \sqrt{-A_5/A_4\alpha_i}$. Since $-A_5/A_4$ is an element of $\mathbb{F}_q$ and  hence a square in $\mathbb{F}_{q^4}$, the roots of $h(-x)$ are in the same orbit under the Frobenius map if and only if $\alpha_i$ is not a square in $\mathbb{F}_{q^4}$, that is the polynomial $x^8+2A_3x^2 +2 A_4\in\mathbb F_q[x]$ has no roots in $\mathbb{F}_{q^4}$. 

\item Let $F(x)$ be as in Case $iv)$ of Proposition \ref{exceptional}.
The polynomial $h_F(-x)$ reads
$$(x^4+2\alpha x^3+2A_3/\alpha  x^2+2(A_3+2\alpha  A_2)x+A_3\alpha   + A_3^2/A_2 + 2A_5\alpha /A_2 + A_2^2)\cdot$$
$$\cdot(x^4+\alpha x^3+A_3/\alpha  x^2+2(A_3+\alpha A_2)x+2A_3\alpha   + A_3^2/A_2 + A_5\alpha /A_2 + A_2^2),$$
where $\alpha^2=2A_2$. 
Hence, the roots of $h_F(-x)$ are in a unique orbit under the Frobenius map if and only if $h_F(-x)$ is irreducible over $\mathbb F_q$.
\end{itemize}
\end{proof}

\begin{remark}
We give two families of good exceptional polynomials arising from Proposition {\rm\ref{goodexceptional}}.
Let $q=3^h$ with $h$ even, and $\eta$ be an odd number; by {\rm \cite[Theorem 3.75]{LN}}, the polynomial $x^8+2\zeta_{q-1}^\eta\in\mathbb F_q[x]$ is irreducible over $\mathbb F_q$.
Therefore, by Case $ \mathit{i)}$ in Proposition {\rm\ref{goodexceptional}}, the polynomial $F(x)=x^9+2\zeta_{q-1}^\eta x$ is good exceptional over $\mathbb F_q$.
Also, by Case $\mathit{ii)}$ in Proposition {\rm\ref{goodexceptional}}, the polynomial
$$F(x)=x^9+\zeta_{q-1}^\eta x^5 + ax^4 + \frac{a^2}{\zeta_{q-1}^\eta}x^3 + 2
\frac{a^3}{\zeta_{q-1}^{2\eta}}x^2 + \left(\zeta_{q-1}^{2\eta}+2\frac{a^4}{\zeta_{q-1}^{3\eta}}\right)x $$
is good exceptional over $\mathbb F_q$, for any $a\in\mathbb F_q$.
\end{remark}

\begin{corollary}\label{coro2}
Assume that $q=3^r>8^4$. The monomial $b^{-1}x^{\frac{q^8-1}{q-1}+1}$  is a CPP of $\mathbb F_{q^8}$ if and only if $b$ is, up to a scalar multiple in $\mathbb{F}_q^*$, a root of some $(F(-x+e)-F(e))/(-x)\in\mathbb F_q[x]$, where $e\in \mathbb{F}_q$ and $F(x)\in\mathbb F_q[x]$ satisfies Case $i)$, $ii)$, or $iii)$ in Proposition {\rm \ref{goodexceptional}}. 
\end{corollary}
\section{Acknowledgements}

The research of D. Bartoli, M. Giulietti, and G. Zini
was partially supported by Ministry for Education, University
and Research of Italy (MIUR) (Project PRIN 2012 ``Geometrie di Galois e
strutture di incidenza'' - Prot. N. 2012XZE22K$_-$005)
 and by the Italian National Group for Algebraic and Geometric Structures
and their Applications (GNSAGA - INdAM).

The research of L. Quoos was partially supported by CNPq -- Proc. 200434/2015-2.
This work was done while L. Quoos was enjoying a sabbatical at the UniversitÃ\`a degli Studi di Perugia.% left from Universidade Federal do Rio de Janeiro.

Daniele Bartoli and Massimo Giulietti are with\\
Dipartimento di Matematica e Informatica \\ Universit\`a degli Studi di Perugia \\ Via Vanvitelli, 1 - 06123 Perugia - Italy\\ {\em emails:} {daniele.bartoli,massimo.giulietti@unipg.it}\\

             Giovanni Zini is with\\
                          Dipartimento di Matematica e Informatica ``Ulisse Dini'' \\ Universit\`a degli Studi di Firenze \\ Viale Morgagni, 67/A - 50134 Firenze - Italy\\ {\em email: }{gzini@math.unifi.it}\\

													Luciane Quoos is with\\
													Instituto de Matem\'atica\\
													Universidade Federal do Rio de Janeiro\\
													Rio de Janeiro 21941-909 - Brazil\\
													{\em email: } {luciane@im.ufrj.br}.

\end{document}